\documentclass{amsart}
\usepackage[utf8]{inputenc}
\usepackage{amsmath}
\usepackage{amssymb}
\usepackage{amsthm}
\usepackage{color}
\usepackage{graphicx}
\usepackage{tabularx}
\usepackage{physics}
\usepackage{tikz-cd}
\usepackage{subfigure}
\usepackage{soul} % for strikethrough
\usepackage[colorinlistoftodos]{todonotes}
\usepackage[colorlinks=true, allcolors=blue]{hyperref}

\usepackage{mathrsfs}

\newtheorem{theorem}{Theorem}[section]
\newtheorem{lemma}[theorem]{Lemma}
\newtheorem{proposition}[theorem]{Proposition}
\newtheorem{corollary}[theorem]{Corollary}

\newtheorem{remark}[theorem]{Remark}

\def\e{\epsilon}

\DeclareMathOperator{\cA}{{\mathcal A}}
\DeclareMathOperator{\cB}{{\mathcal B}}
\DeclareMathOperator{\cS}{{\mathcal S}}
\DeclareMathOperator{\cT}{{\mathcal T}}
\DeclareMathOperator{\cR}{{\mathcal R}}

\renewcommand{\min}{min}

\usepackage[a4paper,top=3cm,bottom=2cm,left=3cm,right=3cm,marginparwidth=1.75cm]{geometry}

\title{Approximate factorization properties for operator systems}

\author{Roy Araiza}
\address{Department of Mathematics, University of Illinois Urbana-Champaign, Urbana, IL, 61801}
\email{raraiza@illinois.edu}

\author{Larissa Kroell}
\address{Department of Mathematics, University of Waterloo, Waterloo, ON, N2L 3G1}
\email{lkroell@uwaterloo.ca}

\author{Travis Russell}
\address{Department of Mathematics, Texas Christian University, Fort Worth, TX, 76109}
\email{travis.b.russell@tcu.edu}

\author{Thomas Sinclair}
\address{Department of Mathematics, Purdue University, West Lafayette, IN, 47907}
\email{tsincla@purdue.edu}

\date{}

\usepackage{natbib}
\usepackage{graphicx}

\begin{document}

\maketitle

\begin{abstract}
    We show that many of the standard nuclearity properties considered in the literature for the hierarchy of operator system tensor products can be expressed as approximate factorization properties, generalizing the well-known Completely Positive Approximation Property for nuclear C*-algebras due to Choi and Effros and its generalization to nuclear operator systems due to Han and Paulsen.
\end{abstract}

\section{Introduction}

A C*-algebra $\cA$ is \textit{nuclear} if for every other C*-algebra $\cB$ the algebraic tensor product $\cA \otimes \cB$ possesses a unique C*-norm. A foundational result in the theory C*-algebras due to Choi and Effros \cite{CEapprox} (and independently to Kirchberg \cite{KiCPAP}) says that a C*-algebra $\mathcal{A}$ is nuclear if and only if it has the \textit{competely positive approximation property}. In the unital case, this means that there exists a net of integers $\{n_{\lambda}\}$ and unital completely positive (ucp) maps $\varphi_{\lambda}: \cA \to M_{n_{\lambda}}$ and $\psi_{\lambda}: M_{n_{\lambda}} \to \cA$ such that $\psi_{\lambda} \circ \varphi_{\lambda}(x) \to x$ in norm for every $x \in \cA$.

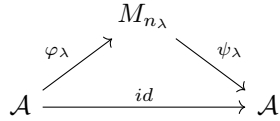
\begin{figure}[h!]
    \centering
        \begin{tikzcd}
        & M_{n_{\lambda}} \arrow[rd, "\psi_{\lambda}" black] \\
        \mathcal{A} \arrow[rr, "id"] \arrow[ru, "\varphi_{\lambda}"]  & & \mathcal{A}
        \end{tikzcd}
    \caption{Completely positive approximation property (CPAP)}
    \end{figure}

Through the work of many authors the tensor theory of \textit{operator systems} has been developed and rigorously studied. Operator systems are unital self-adjoint subspaces of C*-algebras abstractly characterized as matrix-ordered $*$-vector spaces possessing an Archimedean matrix order unit. Given two operator system $\cS$ and $\cT$, an \textit{operator system tensor product} is a matrix ordering on the algebraic tensor product $\cS \otimes \cT$ making it into an operator system with Archimedean matrix order unit $I_{\cS} \otimes I_{\cT}$. A variety of tensor product structures for operator systems have been developed, elucidating deep connections between those tensor products and important properties of operator systems and C*-algebras such as the \textit{exactness}, the \textit{local lifting property}, and the \textit{weak expectation property}, to name a few. These connections led to operator system reformulations of many important problems in C*-algebra theory, for instance Kirchberg's conjecture and the Smith-Ward problem \cite{Kavruk}.

An operator system $\cS$ is called \textit{nuclear} if for every operator system $\cT$, $\cS \otimes \cT$ possesses a unique matrix ordering making $\cS \otimes \cT$ into an operator system with its natural order unit. The completely positive approximation property was generalized to the context of \textit{operator systems} by Han and Paulsen. Specifically, they showed that an operator system $\cS$ is nuclear if and only if it possesses the \textit{operator system completely positive approximation property}. This means that there exists a net of integers $\{n_{\lambda}\}$ and ucp maps $\varphi_{\lambda}: \cS \to M_{n_{\lambda}}$ and $\psi_{\lambda}: M_{n_{\lambda}} \to \cS$ such that $\psi_{\lambda} \circ \varphi_{\lambda}(x) \to x$ in norm for every $x \in \cS$. Han and Paulsen provided an original proof of this property which applies both to operator systems and to unital C*-algebras.

In this paper, we show that a broad class of the operator system nuclearity properties studied in the literature is equivalent to natural ``triangular'' approximate factorization properties. Specifically, we formulate approximate factorization properties equivalent to several central nuclearity properties, including (min,c), (min, el), (min,er), (el,c), (el,max), and (c,max)-nuclearity. We also consider the notion of \textit{self-nuclearity}. We say that an operator system is $(\sigma, \tau)$-self-nuclear if $\cS \otimes_{\sigma} \cS = \cS \otimes_{\tau} \cS$. We prove that, for finite-dimensional operator systems, (min,c) and (min,el)-self-nuclearity are equivalent to approximate factorization properties relating $\cS$ and its operator system dual $\cS^d$. We summarize some of these approximate factorization properties in Table \ref{tab: Fac Props} above.

\begin{table}[h!]
    \centering
    \begin{tabularx}{\textwidth}{|X|c|c|}
        \hline \textbf{Approximation Property} & \textbf{Nuclearity} & \textbf{Diagram} \\ \hline
        CPAP (Han-Paulsen): Point-norm approximation of the identity map through matrix algebras. & (min,max) &  \begin{tikzcd}[ampersand replacement=\&]
        \& M_{n_{\lambda}} \arrow[rd, "\psi_{\lambda}" black] \\
        \mathcal{S} \arrow[rr, "id"] \arrow[ru, "\varphi_{\lambda}"]  \& \& \mathcal{S}
        \end{tikzcd}
        \\ \hline
        Injective envelope AP: Point-norm approximation of the inclusion $\cS \to I(\cS)$ via matrix algbras. & (min,el) & \begin{tikzcd}[ampersand replacement=\&]
        \& M_{n_{\lambda}} \arrow[rd, "\psi_{\lambda}" black] \\
        \mathcal{S} \arrow[rr, "i"] \arrow[ru, "\varphi_{\lambda}"]  \& \& I(\mathcal{S})
        \end{tikzcd} \\ \hline
        Minimal AP: Point-norm approximation of the inclusion $E \to \cS$ of every finite-dimensional subsystem $E$ of $\cS$ via $k$-minimal operator systems $G_{\lambda}$. & (min,el) & \begin{tikzcd}[ampersand replacement=\&]
        \& G_{\lambda} \arrow[rd, "\psi_{\lambda}" black] \\
        E \arrow[rr, "i"] \arrow[ru, "\varphi_{\lambda}"]  \& \& \mathcal{S}
        \end{tikzcd}\\ \hline
        Maximal AP: Point-norm approximation of the inclusion $E \to \cS$ of every finite-dimensional subsystem $E$ of $\cS$ via $k$-maximal operator systems $F_{\lambda}$. & (min,er) & 
        \begin{tikzcd}[ampersand replacement=\&]
        \& F_{\lambda} \arrow[rd, "\psi_{\lambda}" black] \\
        E \arrow[rr, "i"] \arrow[ru, "\varphi_{\lambda}"]  \& \& \mathcal{S}
        \end{tikzcd} \\ \hline
        Universal cover AP: Point-norm approximation of the inclusion $\cS \to C^*_u(\cS)$ via matrix algebras. & (min,c) & \begin{tikzcd}[ampersand replacement=\&]
        \& M_{n_{\lambda}} \arrow[rd, "\psi_{\lambda}" black] \\
        \mathcal{S} \arrow[rr, "i"] \arrow[ru, "\varphi_{\lambda}"]  \& \& C^*_u(\mathcal{S})
        \end{tikzcd} \\
        \hline
        Double-commutant AP: Point-weak approximation of the inclusion $\cS \to C^*_u(\c)$ via matrix algebras $I_{\lambda}$. & (min,c) & \begin{tikzcd}[ampersand replacement=\&]
        \& M_{n_{\lambda}} \arrow[rd, "\psi_{\lambda}" black] \\
        \mathcal{S} \arrow[rr, "i"] \arrow[ru, "\varphi_{\lambda}"]  \& \& \mathcal{S}''
        \end{tikzcd} \\
        \hline
        Injective double-commutant AP: Point-weak approximation of the inclusion $\cS \to \cS''$ via injective operator systems $I_{\lambda}$. & (el,c) & \begin{tikzcd}[ampersand replacement=\&]
        \& I_{\lambda} \arrow[rd, "\psi_{\lambda}" black] \\
        \mathcal{S} \arrow[rr, "i"] \arrow[ru, "\varphi_{\lambda}"]  \& \& \mathcal{S}''
        \end{tikzcd} \\ \hline
        Injective bidual AP: Point-weak approximation of the inclusion $\cS \to \cS^{**}$ via injective operator systems $I_{\lambda}$. & (el,max) & \begin{tikzcd}[ampersand replacement=\&]
        \& I_{\lambda} \arrow[rd, "\psi_{\lambda}" black] \\
        \mathcal{S} \arrow[rr, "i"] \arrow[ru, "\varphi_{\lambda}"]  \& \& \mathcal{S}^{**}
        \end{tikzcd} \\ \hline
        C*-Bidual AP: Point-weak approximation of the inclusion $\cS \to \cS^{**}$ via C*-algebras $\cA_{\lambda}$. & (c,max) & \begin{tikzcd}[ampersand replacement=\&]
        \& \cA_{\lambda} \arrow[rd, "\psi_{\lambda}" black] \\
        \mathcal{S} \arrow[rr, "i"] \arrow[ru, "\varphi_{\lambda}"]  \& \& \mathcal{S}^{**}
        \end{tikzcd} \\ \hline
        Approximation of cp maps $u: \cS^d \to M_n(\cS)$ via matrix algebras into $C^*_u(\cS)$. & Self-(min,c) & \begin{tikzcd}[ampersand replacement=\&]
        \& \cA_{\lambda} \arrow[rd, "\psi_{\lambda}" black] \\
        \mathcal{S}^d \arrow[rr, "u"] \arrow[ru, "\varphi_{\lambda}"]  \& \& M_n(C^*_u(\mathcal{S}))
        \end{tikzcd} \\ \hline
        Approximation of cp maps $u: \cS^d \to M_n(\cS)$ via matrix algebras into $I(\cS)$. & Self-(min,el) & \begin{tikzcd}[ampersand replacement=\&]
        \& \cA_{\lambda} \arrow[rd, "\psi_{\lambda}" black] \\
        \mathcal{S}^d \arrow[rr, "u"] \arrow[ru, "\varphi_{\lambda}"]  \& \& M_n(I(\mathcal{S}))
        \end{tikzcd} \\ \hline
    \end{tabularx}
    \caption{Approximate factorization properties for various operator system nuclearities.}
    \label{tab: Fac Props}
\end{table}

Several of these approximate factorization properties are natural from the perspective of the tensor theory of operator systems. For example, if $\dim(\cS) < \infty$, then our minimal and maximal completely positive approximate factorization properties are analogous to known results concerning finite-dimensional matrix convex sets \cite{PasserPaulsen}. However, we are not aware of these formulations in the literature in this generality. For instance, our maximal CPAP for infinite-dimensional operator systems yields a characterization of the operator system local lifting property in terms of approximate factorizations through k-maximal operator systems. Some of the approximate factorization properties follow via combining the Han-Paulsen approximation theorem with known nuclearity characterizations due to Kavruk and collaborators. Others requiere additional arguments. For example, our proof for the minimal CPAP relies on our Lemma \ref{lem: inverse approx cp} which allows us to perturb the range of a ucp map on a finite-dimensional operator system, while our proof of the maximal CPAP relies on the equivalence of the local lifting property and CP-stability due to Goldbring and the last author \cite{GoldbringSinclair}. Using the CPAP to universal, we are able to recover a result from \cite{KPTT1} that a graph operator system is (min,c)-nuclear if and only if an associated matrix has a positive completion with entries in the universal C*-cover. We also show, using work of Le Merdy on semidiscreteness \cite{LeMer}, that the injective operator systems $I_{\lambda}$ in the injective double commutant approximation property (see Table \ref{tab: Fac Props}) can be taken to be finite-dimensional if and only if the operator system is (min,c)-nuclear. Thus, for a non-exact (el,c)-nuclear operator system, the injective systems appearing in this approximation property cannot, in general, be replaced by matrix algebras. It also gives a point-weak approximate factorization formulation of (min,c)-nuclearity that is not an immediate consequence of the standard approximation formulations. Our paper concludes with approximation properties characterizing ($\sigma$,$\tau$)-\textit{self-nuclearity}, i.e. the situation when $\cS \otimes_{\sigma} \cS = \cS \otimes_{\tau} \cS$. This kind of nuclearity was studied by Kavruk \cite{Kavruk}, who showed, that the equaltiy $\mathcal S_2 \otimes_{\min} \mathcal S_2 =  S_2 \otimes_{c} \mathcal S_2$ for the five-dimensional operator system generated by the two free unitaries is equivalent to Kirchberg's conjecture, equivalently the Connes' embedding problem.

Our paper is organized in the following way. In Section \ref{sec: prelim}, we describe preliminary results concerning operator system tensor products and known equivalences for their nuclearity properties. In Section \ref{sec: Fac Props}, we introduce our approximate factorization properties and prove the various equivalences to operator system nuclearity. Section \ref{sec: Fac Props} is divided into subsections, each focusing on closely related approximate factorization properties.

\section{Preliminaries} \label{sec: prelim}

We assume basic familiarity with operator systems and completely positive maps, and refer the reader to \cite{Pnbook} for more background. For completeness, we review many definitions and results from the literature relevant to this paper.

\subsection{C*-covers}

Given an operator system $\mathcal{S}$, a C*-cover for $\mathcal{S}$ consists of a C*-algebra $\mathcal{A}$ and a unital complete order embedding $j: \mathcal{S} \to \mathcal{A}$ such that $C^*(j(\mathcal{S})) = \mathcal{A}$. Usually the embedding $j$ is clear from context and not mentioned when describing a C*-cover.

Among all C*-covers, two will be especially important: the \textit{universal C*-cover} $C^*_u(\cS)$ and the \textit{C*-envelope} $C^*_e(\cS)$. We will also be interested in the \textit{injective envelope}, which is the smallest injective operator system $I(\cS)$ containing a unital completely isometric copy of $\cS$. It is known that $I(\cS)$ with the Choi--Effros product becomes a C*-algebra where the C*-subalgebra generated by $\cS$ in $I(\cS)$ coincides with $C^*_e(\cS)$.

The universal C*-cover is characterized by the following universal property: whenever $\varphi: \cS \to B(H)$ is ucp, there exists a unique $*$-homomorphism $\pi_{\varphi}: C^*_u(\cS) \to B(H)$ extending $\varphi$. The C*-algebra $C^*_u(\cS)$ was first studied by Kirchberg and Wasserman in \cite{KirWas}. It may be identified with the completion of a certain Fock algebra over $\cS$ (e.g. see \cite{KPTT1}) or as the C*-algebra generated by the image of $\cS$ under the direct sum of ucp maps from $\cS$ into matrix algebras (see \cite{KirWas}). This characterization shows, as in \cite{KirWas}, that $C_u^*(\mathcal S)$ is residually finite-dimensional. Moreover, its universal property guarantees that every C*-cover of $\cS$ is a quotient of $C^*_u(\cS)$.

\begin{figure}[h!]
\centering
\begin{tikzcd}
C^*_u(\cS) \arrow[rd, "\pi_{\varphi}"] \\
\mathcal \mathcal{S} \arrow[u, hook] \arrow[r, "\varphi" black] & B(H)
\end{tikzcd}
\caption{The universal property of $C^*_u(\cS)$.}
\end{figure}

The C*-envelope of an operator system is characterized by a co-universal property: whenever $\cA$ is a C*-cover of $\cS$, there exists a unital $*$-homomorphism $\pi: \cA \to C^*_e(\cS)$ fixing the copy of $\cS$ in both C*-covers (i.e. extending the embedding of $\cS$ into $C^*_e(\cS)$). This co-universal property merely states that $C^*_e(\cS)$ is the unique quotient of every C*-cover for $\cS$. The existence of the C*-envelope was conjectured by Arveson in \cite{Arv}, and later verified by Hamana \cite{Hama} in through the existence of the injective envelope. The original approach suggested by Arveson was verified much later by means of a deeper results (see \cite{DM}, \cite{ArvChoquet}, \cite{DK}).

\begin{figure}[h!]
\centering
\begin{tikzcd}
\cA \arrow[rd, "\pi"] \\
\mathcal \mathcal{S} \arrow[u, hook] \arrow[r, "i" black, hook] & C^*_e(\cS)
\end{tikzcd}
\caption{The universal property of $C^*_e(\cS)$.}
\end{figure}
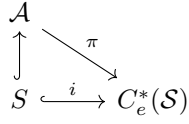

An operator system $\mathcal{I}$ is \textit{injective} if whenever $\cS \subseteq \cT$ is a complete order embedding and $\varphi: \cS \to \mathcal{I}$ is ucp, then there exists a ucp extension $\varphi': \cT \to \mathcal{I}$. The \textit{injective envelope} $I(\cS)$ is the minimal injective operator system contain $\cS$ completely isometrically. Specifically, whenever $\cR$ is an injective operator system and $j: \cS \to \cR$ is a unital complete order embedding, there exists an injective subsystem $\mathcal{I}$ such that $j(\cS) \subseteq \mathcal{I} \subseteq \cR$ and such that for injective operator system $j(\cS) \subseteq \mathcal{J} \subseteq \mathcal{I}$ we must have $\mathcal{J} = \mathcal{I}$. The injective envelope is unique in the sense that that if $\mathcal{I}'$ is another manifestation of the injective envelope as above, possibly with a different map $j' : \cS \rightarrow \cR'$ for an injective operator system $\cR'$, then $\mathcal{I} \cong \mathcal{I}'$ via a unital complete isometry fixing both copies of $\cS$. Therefore, we shall denote the injective envelope of $\cS$ by $I(\cS)$. The existence and uniqueness of $I(\cS)$ was shown by Hamana, who also showed that with the Choi--Effros product $I(\cS)$ is completely isometric to a C*-algebra and that the C*-subalgebra of $I(\cS)$ generated by $\cS$ is the C*-envelope $C^*_e(\cS)$. 

\subsection{Tensor products}

Let $\mathcal{S}$ and $\mathcal{T}$ be operator systems with units $e$ and $f$, respectively. A tensor product structure $\mathcal{S} \otimes \mathcal{T}$ is a matrix ordering $\{C_n\}$ on $\mathcal{S} \otimes \mathcal{T}$ such that
\begin{enumerate}
    \item $(\mathcal{S} \otimes \mathcal{T}, \{C_n\}, e \otimes f)$ is an operator system,
    \item $P \otimes Q \in C_{pq}$ for every $P \in M_p(\mathcal{S})^+$ and $Q \in M_q(\mathcal{T})^+$, and
    \item $\varphi \otimes \psi: \mathcal{S} \otimes \mathcal{T} \to M_{pq}$ is completely positive for every ucp $\varphi: \mathcal{S} \to M_p(\mathbb{C})$ and $\psi: \mathcal{T} \to M_q$.
\end{enumerate}

An \textit{operator system tensor product} $\otimes_{\tau}$ is a mapping that assigns to any pair of operator systems $\cS$ and $\cT$ a corresponding operator system tensor product structure on $\cS \otimes \cT$, denoted $\cS \otimes_{\tau} \cT$. If this mapping is a functor on the category of operator systems (with ucp maps as morphisms), then the tensor product is called \textit{functorial}. Specifically, we say that $\otimes_{\tau}$ is functorial if for any ucp maps $\varphi: \cS \to \cS'$ and $\psi: \cT \to \cT'$, it follows that $\varphi \otimes \psi: \cS \otimes_{\tau} \cT \to \cS' \otimes_{\tau} \cT'$ is ucp. A tensor product is \textit{left injective} (resp. \textit{right injective}) if $\varphi \otimes id: \cS \otimes_{\tau} \cT \to \cS' \otimes \cT$ (resp. $id \otimes \varphi: \cS \otimes_{\tau} \cT \to \cS \otimes \cT'$) is a unital complete order embedding whenever $\varphi$ (resp. $\psi$) is a unital complete order embedding. If $\otimes_{\tau}$ is both left and right injective, it is called \textit{injective}. Similarly, we may define \textit{left projective}, \textit{right projective}, and \textit{projective} tensor products if the tensor product preserves complete quotient maps on the left factor, right factor, or both factors. A tensor product is \textit{associative} if $\cS \otimes_{\tau} (\cT \otimes_{\tau} \cR) \cong (\cS \otimes_{\tau} \cT) \otimes_{\tau} \cR$ for all operator systems $\cS$, $\cT$, and $\cR$. It is \textit{symmetric} if the map $x \otimes y \mapsto y \otimes x$ from $\cS \otimes_{\tau} \cT$ to $\cT \otimes_{\tau} \cS$ is a complete order isomorphism.

The following tensor products were all defined in \cite{KPTT1}. In the following descriptions, fix operator system $\cS$ and $\cT$ with units $e$ and $f$, respectively.

\begin{itemize}
    \item \textbf{The minimal tensor product}: We say that $x \in M_n(\cS \otimes_{min} \cT)^+$ provided that $(\varphi \otimes \psi) ^{(n)}(x) \geq 0$ for every $\varphi: \cS \to M_p$ and $\psi: \cT \to M_q$ for all $p,q \in \mathbb{N}$. The minimal tensor product is known to be functorial, injective, associative, and symmetric. It is also completely isometrically isomorphic with the operator space minimal tensor product of $\cS$ and $\cT$ as operator spaces (see e.g. Corollary 4.9 of \cite{KPTT1}). Finally, injectivity of the minimal tensor product is often described concretely as follows: whenever $\pi:\cS \subseteq B(H)$ and $\rho:\cT \subseteq B(K)$ are unital complete order embeddings, then $\pi \otimes \rho: \cS \otimes_{min} \cT \subseteq B(H \otimes K)$ is a unital complete order embedding.

    \item \textbf{The maximal tensor product}: We say that $x \in M_n(\cS \otimes_{max} \cT)^+$ provided that for every $\epsilon > 0$ there exist $P \in M_p(\cS)^+$ and $Q \in M_q(\cT)^+$ and a scalar matrix $\alpha$ such that $x + \epsilon I_n \otimes e \otimes f = \alpha^* (P \otimes Q) \alpha$. Equivalently, the operator system structure on $\cS \otimes_{max} \cT$ is the largest (i.e. smallest matrix ordering) making $\cS \otimes_{max} \cT$ an operator system and including all elementary tensors $P \otimes Q$ of positive matrices within the positive cone. The maximal tensor product is known to be functorial, projective, associative, and symmetric. Moreover, the maximal tensor product coincides with the maximal tensor product of C*-algebras whenever both factors are C*-algebras, i.e. $\cA \otimes_{max} \cB = \cA \otimes_{C^*-max} \cB$ for all C*-algebras $\cA$ and $\cB$.

    \item \textbf{The commuting tensor product}: Suppose that $\varphi: \cS \to B(H)$ and $\psi: \cT \to B(H)$ are ucp. We define $\varphi \cdot \psi: \cS \otimes \cT \to B(H)$ by setting $\varphi \cdot \psi (s \otimes t) = \varphi(s) \psi(t)$ for all $s \in \cS$ and $t \in \cT$ and then extending by linearity to $\cS \otimes \cT$. We say that $x \in M_n(\cS \otimes_c \cT)$ if $(\varphi \cdot \psi)^{(n)}(x) \geq 0$ for all pairs of maps $\varphi: \cS \to B(H)$ and $\psi: \cT \to B(H)$ with commuting range. The commuting tensor product is known to be functorial and symmetric, but not left or right injective. It is an open problem whether or not the commuting tensor product is associative. Furthermore, whenever $\cA$ is a C*-algebra, we have $\cS \otimes_c \cA = \cS \otimes_{max} \cA$. Finally, we have the inclusions $\cS \otimes_c \cT \subseteq C^*_u(\cS) \otimes_{max} \cT$ and $\cS \otimes_c \cT \subseteq \cS \otimes_{max} C^*_u(\cT)$ (see \cite{KPTT2}).

    \item \textbf{The injective tensor products}: We define the \text{left injective} tensor product $\otimes_{el}$ by the following inclusion: $\cS \otimes_{el} \cT \subseteq I(\cS) \otimes_{max} \cT$, i.e. an element is positive in $M_n(\cS \otimes_{el} \cT)^+$ if it is positive as an element of $M_n(I(\cS) \otimes_{max} \cT)^+$. Similarly, we define the \text{right injective} tensor product $\otimes_{er}$ by the inclusion $\cS \otimes_{er} \cT \subseteq \cS \otimes_{max} I(\cT)$ and the \textit{injective} tensor product $\otimes_e$ by the inclusion $\cS \otimes_e \cT \subseteq I(\cS) \otimes_{max} I(\cT)$. All of the injective tensor products are known to be functorial. The left (resp. right) injective tensor product is left (resp. right) injective, and the injective tensor product is injective. The left injective and right injective tensor products are not symmetric, although the injective tensor product is symmetric. We also have $\cS \otimes_{el} \cT \cong \cT \otimes_{er} \cS$.
\end{itemize}

A partial order can be imposed on the set of tensor products as follows: we say $\tau \preceq \sigma$ if the identity map from $\cS \otimes_{\sigma} \cT \to \cS \otimes_{\tau} \cT$ is ucp. For the above defined tensor products, we have $min \preceq e \preceq el, er \preceq c \preceq max$ (with $el$ and $er$ not comparable). Moreover, the following universal properties hold: if $\tau$ is any tensor product, then $min \preceq \tau \preceq max$. Moreover, if $\tau$ is left (resp. right) injective, then $\tau \preceq el$ (resp. $\tau \preceq er$), and $\tau$ is injective then $\tau \preceq e$.

\subsection{$k$-minimality and $k$-maximality}

A useful fact about operator systems is the following: an element $x \in M_k(\cS)$ is positive if and only if $\varphi^{(k)}(x) \geq 0$ for all ucp $\varphi: \cS \to M_k$. An operator system $\cS$ is \textit{$k$-minimal} if whenever $\varphi^{(n)}(x) \geq 0$ for every ucp $\varphi: \cS \to M_k$, it follows that $x \in M_n(\cS)^+$. The following theorem collects standard characterizations of k-minimal operator systems, due largely to Xhabli \cite{Xhabli}, together with the formulation in item (4) from \cite{ART23}.

\begin{theorem}
    Let $\cS$ be an operator system. The following statements are equivalent.
    \begin{enumerate}
        \item $\cS$ is $k$-minimal.
        \item There exists a compact Hausdorf space $X$ and a unital complete order embedding $j: \cS \to C(X,M_k)$.
        \item For every operator system $\cT$ and $k$-positive map $\varphi: \cT \to \cS$, $\varphi$ is completely positive.
        \item Whenever $x \in M_n(\cS)$ and $\alpha^* x \alpha \in M_k(\cS)^+$ for all $\alpha \in M_{n,k}$, it follows that $x \geq 0$ (see \cite{ART23}).
    \end{enumerate}
\end{theorem}

From (2), it is evident that $M_n$ is $k$-minimal as an operator system whenever $n \leq k$. The same holds for any $l^{\infty}$ direct sum of matrix algebras $\oplus_{\lambda} M_{n_{\lambda}}$ with $n_{\lambda} \leq k$ for all $k$, or any subsystem of such a matrix algebra. The following was shown in \cite{ART23}:

\begin{theorem}
    Let $\cS$ be a $k$-minimal operator system. If $\varphi: C^*_e(\cS) \to B(H)$ is irreducible, then $\dim(H) \leq k$.
\end{theorem}

Since any C*-algebra may be represented faithfully into the $l^{\infty}$-direct sum of images of its irreducible representations, the C*-envelope embeds isometrically into an $l^{\infty}$ direct sum of matrix algebras of size no more than $k \times k$. As such a direct sum is injective, we conclude that the injective envelope $I(\cS)$ is $k$-minimal. In fact, $I(\cS)$ is itself isometric to an $l^{\infty}$ direct sum of matrix algebras no larger than $k \times k$. See also \cite{Kavruk} for more discussion on $k$-minimal operator systems. 

In duality with $k$-minimal operator systems, we have the family of \textit{$k$-maximal} operator systems. A $k$-minimal operator system is an operator system $\cS$ for which the matrix cone $\{M_n(\cS)^+\}$ is the smallest matrix cone generated by $M_k(\cS)^+$. More explicitly, $M_n(\cS^{k-max})^+$ is the archimedean closure of the cone of elements of the form $\alpha^* \text{diag}(P_1, P_2, \dots, P_m) \alpha$ where $P_1, \dots, P_m \in M_k(\cS)^+$ and $\alpha \in M_{mk,n}$. The $k$-maximal operator systems are also characterized the by following property: an operator system $\cS$ is $k$-maximal if and only if every unital $k$-positive map $\varphi: \cS \to \cT$ from $\cS$ to an operator system $\cT$ is ucp \cite{Xhabli}. 

\subsection{Duality}

For any operator system $\cS$, the space $\cS^d$ of bounded linear maps from $\cS$ to $\mathbb{C}$ can be naturally viewed as a matrix ordered $*$-vector space. The adjoint of a linear functional $\varphi: \cS \to \mathbb{C}$ is defined by $\varphi^*(x) := \varphi(x^*)^*$, and the positive cone $C_n^d$ of $M_n(\cS^d)$ is identified with the set of completely positive maps $\varphi: \cS \to M_n$. In the case when $\cS$ is finite-dimensional, Choi and Effros show that the positive cone of $\cS^d$ always has a (non-unique) interior point $\varphi_0$ making $(\cS^d, \{C_n^d\}, \varphi_0)$ into an abstract operator system. Although the positive cone is unique, the induced matrix norms depend on the choice of order unit, making this operator system structure non-unique.

Given finite-dimensional operator system $\cS$ and $\cT$, we have the following identifications (\cite{FP}, \cite{Kavruk}): \[ (\cS \otimes_{min} \cT)^d = \cS^d \otimes_{max} T^d \quad \text{ and } \quad (\cS \otimes_{max} \cT)^d = \cS^d \otimes_{min} \cT^d. \] Moreover, whenever $\cT$ is finite-dimensional, we may identify the positive cone of $M_n(\cS \otimes_{min} \cT^d) \cong M_n(\cS) \otimes_{min} \cT^d$ with the set of completely positive maps from $\cT$ to $M_n(\cS)$ via the identification of the elementary tensor $(s_{ij}) \otimes t^d$ with the map $\varphi: \cT \to M_n(\cS)$ defined by $\varphi(t) = t^d(t) \cdot (s_{ij})$.

We will also make use of the following observations by Kavruk \cite{Kavruk}, which states that for any integer $k \in \mathbb{N}$ and finite-dimensional operator system $\cS$, we have the identifications $(S^{k-max})^d = (S^d)^{k-min}$, and $(S^{k-min})^d = (S^d)^{k-max}$. Hence, the dual of a $k$-minimal operator system is $k$-maximal, and the dual of a $k$-maximal operator system is $k$-minimal.

\subsection{Nuclearity properties}

Suppose $\tau$ and $\sigma$ are two operator system tensor products. An operator system $\cS$ is said to be $(\tau, \sigma)$-nuclear if $\cS \otimes_{\tau} \cT = \cS \otimes_{\sigma} \cT$ for every operator system $\cT$. Nuclearity of operator systems has been studied by many authors over the past 15 years (see e.g. \cite{KPTT1}, \cite{KPTT2}, \cite{Kavruk}, \cite{HP}, \cite{Htp}), where many nuclearity properties have been linked to approximate factorization properties which we summarize here.

\begin{itemize}
    \item An operator system $\cS$ is said to have the \textit{completely positive approximation property} provided that there exists a net of maps $\varphi_{\lambda}: \cS \to M_{n_{\lambda}}$, $\psi_{\lambda}: M_{n_{\lambda}} \to \cS$ such that $\psi_{\lambda} \circ \varphi_{\lambda} \to id$ in point-norm. This property was shown by Han and Paulsen in \cite{HP} to be equivalent to $(min,max)$-nuclearity of $\cS$, and generalized the same result for C*-algebras studied by Choi and Effros in \cite{CEapprox} and Kirchberg in \cite{KiCPAP}. The class of (min,max)-nuclear operator systems includes that of nuclear C*-algebras such as matrix algebras $M_n, n \in \mathbb N$, and commutative C*-algebras. Furthermore, Tsirelson's non-commutative 1-cube $NC(1)$ is a (min,max)-nuclear operator system \cite{FKPT1}.
    
    \begin{figure}[h!]
    \centering
        \begin{tikzcd}
        & M_{n_{\lambda}} \arrow[rd, "\psi_{\lambda}" black] \\
        \mathcal{S} \arrow[rr, "id"] \arrow[ru, "\varphi_{\lambda}"]  & & \mathcal{S}
        \end{tikzcd}
    \caption{Completely positive approximation property}
    \end{figure}
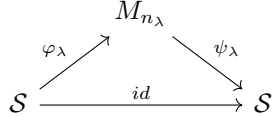

    \item An operator system $\cS$ is said to have the \textit{weak expectation property} (or WEP) provided that the inclusion of $\cS$ into its bidual $\cS^{**}$ extends to a ucp map $\varphi: I(\cS) \to \cS^{**}$. This is equivalent to $\cS$ being $(el,max)$-nuclear. One implication was shown in \cite{KPTT2}, while the other was shown in \cite{Htp}

    \begin{figure}[h!]
    \centering
        \begin{tikzcd}
        I(\mathcal{S}) \arrow[dashrightarrow, rd, "\varphi" black] \\
        \mathcal \mathcal{S} \arrow[u, hook] \arrow[r, "i" black, hook] & \mathcal{S}^{**}
        \end{tikzcd}
    \caption{The weak expectation property (WEP).}
    \end{figure}
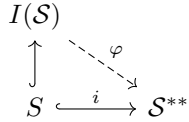

    \item An operator system $\cS$ is said to have the \textit{double commutant expectation property} (or DCEP) provided that every complete order embedding $\pi: \cS \to B(H)$ extends to a ucp map $\varphi: I(\cS) \to \pi(\cS)''$. This property was introduced in \cite{KPTT2} where it was shown to be equivalent to $(el,c)$-nuclearity.

    \begin{figure}[h!]
    \centering
        \begin{tikzcd}
        I(\mathcal{S}) \arrow[dashrightarrow, rd, "\varphi" black] & \\
        \mathcal \mathcal{S} \arrow[u, hook] \arrow[r, "\pi" black] & \pi(\mathcal{S})'' \arrow[r, hook] & B(H)
        \end{tikzcd}
    \caption{The double commutant expectation property (DCEP).}
    \end{figure}

    \item An operator system $\cS$ is said to have the the \textit{local lifting property} (or LLP) if for every finite-dimensional subsystem $\cS_0 \subseteq \cS$ and every ucp map $\varphi: \cS \to \cA/I$ (where $\cA$ is a C*-algebra and $I \subseteq \cA$ is a two-sided closed ideal), the restriction of $\varphi$ to $\cS_0$ admits a ucp lift $\widetilde{\varphi}: \cS_0 \to \cA$. LLP was shown to be equivalent to $(min,er)$-nuclearity in \cite{KPTT2}. Due to a classical result of Kirchberg \cite{KirLLP}, it is well-known that for any free group $F$, the full group C*-algebra $C^*(F)$ has the LLP. In a similar fashion, for $n \geq 2$, let \[ \mathcal S_n:= \text{span}\{e,u_i, u_i^*: i = 1,2,\dots, n\} \subset C^*(\mathbb F_n),\] denote the operator system spanned by the free unitary generators of $C^*(\mathbb F_n)$. Then $S_n$ has the LLP, and thus is (min,er)-nuclear. Furthermore, for the universal $\mathrm{C}^*$-algebra $\mathcal{U}(n)$ of $n^2$-generators $(u_{ij})$ with the restriction that the matrix $(u_{ij})$ is unitary, the operator subsystem
        \[\mathcal V_n = \mathrm{span} \{1, u_{ij}, u_{ij}^* \, \vert \, i,j = 1, \ldots, n \} \subseteq \mathcal{U}(n),\] is (min, er)-nuclear \cite{Harris}.

    \begin{figure}[h!]
    \centering
        \begin{tikzcd}
        & & \cA \arrow[d, "q" ] \\
        \mathcal{S}_0 \arrow[urr, "\widetilde{\varphi}"] \arrow[r, hook] & \mathcal{S} \arrow[r, "\varphi" black] & \cA/I
        \end{tikzcd}
    \caption{The local lifting property (LLP).}
    \end{figure}

    \item An operator system $\cS$ is said to be \textit{exact} if for every C*-algebra $\cA$ and closed two-sided ideal $I \subseteq \cA$ the natural mapping
    \[ (\cS \otimes_{min} \cA) / (\cS \otimes_{min} I) \to \cS \otimes_{min} (\cA /I) \]
    is a complete order isomorphism. It was shown in \cite{KPTT2} that an operator system is exact precisely when it is $(min,el)$-nuclear.
    
\end{itemize}

\subsection{Han-Paulsen Approximate Factorization Theorem}

The proof that (min,max)-nuclearity is equivalent to the CPAP is an application of the following approximate factorization theorem due to Han and Paulsen.

\begin{theorem}[Han-Paulsen] \label{thm: Han-Paulsen}
    Let $\phi: S \to T$ be a ucp map. Then the following are equivalent.
    \begin{enumerate}
        \item For every operator system $\cR$, $id \otimes \phi: \cR \otimes_{min} \cS \to \cR \otimes_{max} \cT$ is ucp.
        \item For every finite dimensional operator system $E$, $id \otimes \phi: E \otimes_{min} \cS \to E \otimes_{max} \cT$ is ucp.
        \item There exist nets of ucp maps $\phi_{\lambda}: \cS \to M_{n_{\lambda}}$ and $\psi_{\lambda}: M_{n_{\lambda}} \to \cT$ such that $\psi_{\lambda} \circ \phi_{\lambda} \to \phi$ in the point-norm topology.
    \end{enumerate}
\end{theorem}

In (3) of Theorem \ref{thm: Han-Paulsen}, nets can be replaced with sequences when $\dim(\cS) < \infty$. This is seen in the ``(2) implies (3)'' portion of the proof from \cite{HP}: the net is defined as pairs $(E,\epsilon)$ where $\epsilon > 0$ and $E \subseteq \cS$ is a finite-dimensional subsystem with partial ordering defined by $(E, \epsilon) \leq (E', \epsilon')$ if and only if $E \subseteq E'$ and $\epsilon \geq \epsilon'$. In the case when $\dim(\cS) < \infty$, it suffices to consider the subsequence $(\cS, \frac{1}{n})$.

\section{Approximate Factorization Properties} \label{sec: Fac Props}

In this section, we describe how most nuclearity properties studied in the literature on operator systems have equivalent formulations as approximate factorization properties. All of these properties involve an approximate factorization of the inclusion of the underlying operator system $\cS$ into various canonical C*-algebras containing $\cS$ through specialized classes of operator systems.

\subsection{(min,el) nuclearity}

We say that an operator system has the \textbf{injective envelope approximation property} (IEAP) if the inclusion $i: \cS \to I(\cS)$ approximately factors through matrix algebras --- i.e. there exists a net $\Lambda$ and ucp maps $\varphi_{\lambda}: \cS \to M_{n_{\lambda}}$ and $\psi_{\lambda}: M_{n_{\lambda}} \to I(\cS)$ such that for every $x \in \cS$, $\psi_{\lambda} \circ \varphi_{\lambda} (x) \to i(x)$. 

\begin{figure}[h!]
\centering
    \begin{tikzcd}
    & M_{n_{\lambda}} \arrow[rd, "\psi_{\lambda}" black] \\
    \mathcal{S} \arrow[rr, "i"] \arrow[ru, "\varphi_{\lambda}"]  & & I(\cS)
    \end{tikzcd}
\caption{The injective envelope approximation property.}
\end{figure}
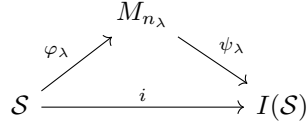

\begin{theorem} \label{t:min-el=injcpap}
    Let $\cS$ be an operator system. The the following statements are equivalent.
    \begin{enumerate}
        \item $\cS$ is (min,el)-nuclear.
        \item $\cS$ is exact.
        \item $\cS$ has the IEAP.
    \end{enumerate}
\end{theorem} 

\begin{proof}
    Let $i: \cS \to I(\cS)$ be the inclusion of $\cS$ into its injective envelope. Then $\cS \otimes_{min} \cR = \cS \otimes_{el} \cR$ for every operator system $\cR$ if and only if the map $i \otimes id: \cS \otimes_{min} \cR \to I(\cS) \otimes_{max} \cR$ is ucp, since the image of $\cS \otimes_{min} \cR$ under $i \otimes id$ is $\cS \otimes_{el} \cR \subseteq I(\cS) \otimes_{max} \cR$. Since the max tensor product is commutative, the statement follows from Theorem \ref{thm: Han-Paulsen}.
\end{proof}

\begin{remark}
    \emph{It is easy to see that $I(\cS)$ in the IEAP can be replaced by $B(H)$, where $\cS$ embeds completely isometrically into $B(H)$, by considering the conditional expectation $\pi: B(H) \to I(\cS)$. In the case when $\cS$ is a C*-algebra, this recovers the well-known property that a C*-algebra is exact if and only if it admits a nuclear embedding into $B(H)$ (c.f. \cite{BrownOzawa}).}
\end{remark}

From the CPAP to IEAP, we can obtain an alternative approximate factorization property that does not involve the injective envelope. We say that an operator system $\cS$ has the \textbf{minimal approximation property} (minimal AP) if for every finite-dimensional operator subsystem $E \subseteq \cS$, the identity map $id: E \to E$ approximately factors through $k$-minimal operator systems --- i.e., there exists a net $\Lambda$ and, for every $\lambda \in \Lambda$, an integer $n_{\lambda} \in \mathbb{N}$, a $n_{\lambda}$-minimal operator system $S_{\lambda}$, and ucp maps $\varphi_{\lambda}: E \to S_{\lambda}$ and $\psi_{\lambda}: S_{\lambda} \to E$ such that for every $x \in E$, $\psi_{\lambda} \circ \varphi_{\lambda} \to i(x)$. In the case when $\cS$ is finite-dimensional, we can take $E = \cS$ and replace the net with a sequence.

\begin{figure}[h!]
\centering
    \begin{tikzcd}
    & S_{\lambda}^{k_{\lambda}-min} \arrow[rd, "\psi_{\lambda}" black] \\
    E \arrow[rr, "id"] \arrow[ru, "\varphi_{\lambda}"]  & & E \subseteq \cS
    \end{tikzcd}
\caption{The minimal approximation property.}
\end{figure}
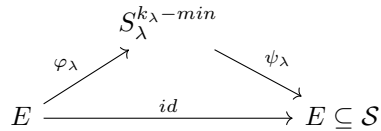

To prove the next result, we need the following technical lemma. Essentially, this lemma says that if $E$ is a finite-dimensional operator system and $\pi: E \to B(H)$ approximates the identity map (where $E \subseteq B(H)$), then $\pi^{-1}$ is approximately ucp on its range. We will use this lemma later to replace an approximately correct range by an honest ucp map into the desired finite-dimensional operator subsystem. 

\begin{lemma} \label{lem: inverse approx cp}
    Let $\cA$ be a C*-algebra and suppose that $E \subseteq \cA$ is a finite-dimensional subsystem and $0 < \epsilon < 1/2$. Suppose that $\pi: E \to \cA$ is an injective ucp map satisfying $\| \pi(x) - x \| < \epsilon$ for every $x \in E$ with $\|x\| \leq 1$. Then there exists a positive linear functional $\eta: \pi(E) \to \mathbb{C}$ such that the map $\rho(x) = \pi^{-1}(x) + \eta(x)I$ is completely positive and such that \[ 0 < \eta(I) \leq C \epsilon \] where $C$ a constant which depends only on the Banach space $(E, \|\cdot\|)$.
\end{lemma}

\begin{proof}
    Let $x_1, \dots, x_d \in E$ be a linearly independent set of self-adjoint elements which span $E$ and satisfy $\|x_i\| = 1$ for each $i$ and $x_1 = I$. Let $y_i = \pi(x_i)$ for each $i$. Then $y_1 = I$ and $1 - \epsilon < \|y_i\| \leq 1$ for each $i > 1$. For each $i$, let $\sigma_i: \pi(E) \to \mathbb{C}$ denote the coordinate function defined on the basis $\{y_1, \dots, y_d\}$ by $\sigma_i(y_j) = \delta_{ij}$, where $\delta_{ij}$ is the Dirac delta function. Similarly, let $\psi_i: E \to \mathbb{C}$ denote the coordinate functionals for the basis $\{x_1, \dots, x_n\}$. Choose $M > 0$ such that $\|\psi_i\| \leq M$ for all $i$. Note that $M$ depends only on the Banach space $(E,\|\cdot\|)$.
    
    Let $y \in \pi(E)$ such that $\|y\| = 1$ and $y=y^*$. Assume $y = \sum_k \alpha_k y_k$ for $\alpha_k \in \mathbb{R}$. Let $x' = \|\sum_k \alpha_k x_k\|^{-1} \sum_k \alpha_k x_k$. Since $\|x'\| = 1$, \[ 1 - \epsilon \leq \|\sum_k \alpha_k x_k\|^{-1} \|\sum \alpha_k y_k \| \leq 1. \] So $1 \leq \| \sum \alpha_k x_k \| \leq (1-\epsilon)^{-1}$. Also \[ |\sigma_i(y)| = |\alpha_i| = |\psi_i(\sum \alpha_k x_k)| \leq M \|\sum \alpha_k x_k \| \leq (1-\epsilon)^{-1} M. \] So we have $\|\sigma_i\| \leq (1-\epsilon)^{-1} M$ for each $i$.

    Recall that $(\pi(E)^d \otimes_{min} \cA)^+ = CP(\pi(E), \cA)$ via the identification of an element $\sum_{i=0}^d \sigma_i \otimes z_i$ with the linear map from $\pi(E)$ to $\cA$ sending $y_i$ to $z_i$. Notice that the identity map on $\pi(E)$ is represented by the tensor $\sum \sigma_i \otimes y_i$ in $\pi(E)^d \otimes_{min} \cA$, and this map is completely positive. 

    For each $i$, since $\sigma_i$ is a bounded self-adjoint functional, there exist positive functionals $\sigma_i^+$ and $\sigma_i^-$ such that $\sigma_i = \sigma_i^+ - \sigma_i^-$ and $\|\sigma_i\| = \|\sigma_i^+\| + \|\sigma_i^-\|$. Define $\varphi := \sum_k \sigma_k^+ + \sigma_k^-$. We note that $\| \varphi \| \leq \sum_k \|\sigma_k^+\| + \|\sigma_k^-\| = \sum_k \|\sigma_k\| \leq d M (1-\epsilon)^{-1}$, and thus $\varphi(I) = \|\varphi\| \leq d M (1-\epsilon)^{-1}$. Now for each $i$ and $\dagger \in \{+,-\}$,
    \[ \pm \left( \sum \sigma_i^{\dagger} \otimes y_i - \sum \sigma_i^{\dagger} \otimes x_i \right) = \sum \sigma_i^{\dagger} \otimes \pm (y_i-x_i) \leq \sum \sigma_i^{\dagger} \otimes (\epsilon I) \leq \epsilon \varphi \otimes I \]
    since $\varphi_i^{\dagger} \leq \varphi$. Since $\sigma_i = \sigma_i^+ - \sigma_i^-$, we conclude that
    \[ \sum \sigma_i \otimes y_i - \sum \sigma_i \otimes x_i \leq 2\epsilon \varphi \otimes I. \]
    Since $\sum \sigma_i \otimes y_i$ corresponds to the identity map on $\pi(E)$, $\sum \sigma_i \otimes x_i$ corresponds to $\pi^{-1}$, and $\varphi \otimes I$ is completely positive, we conclude that $\rho := \pi^{-1} + 2 \epsilon \varphi \otimes I$ is completely positive. Let $\eta = 2\epsilon \varphi$. Then \[ \eta(I) \leq \dfrac{2 d M \epsilon}{1-\epsilon} < 4 d M \epsilon = C \epsilon \] where $C = 4d M$. Since the constant $C$ depends only on the Banach space $(E,\|\cdot\|)$, the proof is complete.
\end{proof}

\begin{theorem} \label{thm: minimal CPAP}
    Let $\cS$ be an operator system. Then $\cS$ has IEAP if and only if it has the minimal AP.
\end{theorem}

\begin{proof}
    First, assume that $\cS$ has the IEAP. Hence we have a net $\Lambda$ and ucp maps $\varphi_{\lambda}: \cS \to M_{n_{\lambda}}$ and $\psi_{\lambda}: M_{n_{\lambda}} \to I(\cS)$ such that for every $x \in \cS$, $\pi_{\lambda}(x) := \psi_{\lambda} \circ \varphi_{\lambda} (x) \to i(x)$. Fix $E \subseteq \cS$. Let $S_{\lambda} := \varphi_{\lambda}(E) \subseteq M_{n_{\lambda}}$. Then $S_{\lambda}$ is $n_{\lambda}$-minimal and finite-dimensional with $\dim(S_{\lambda}) \leq \dim(E)$. Let $E_{\lambda} = \psi_{\lambda}(S_{\lambda}) \subseteq I(\cS)$.

    Assume $0 < \epsilon < 1/2$. Since $\pi_{\lambda}$ converges to the identity map pointwise, and since the unit ball of $E$ is compact, there exists $\lambda_0$ such that $\|\pi_{\lambda}(x) - x \| < \epsilon$ for all $x$ in the unit ball of $E$ and $\lambda > \lambda_0$. Since every unit vector is mapped to a non-zero vector, $\pi_{\lambda}$ is injective. By the Lemma, there exists a faithful linear functional $\eta: E_{\lambda} \to \mathbb{C}$ such that $\eta(I) =: \delta \leq C \epsilon$ and $\pi_{\lambda}^{-1}: E_{\lambda} \to E$ satisfies $\pi_{\lambda}^{-1} + \eta \otimes I$ is completely positive. Hence the map $\rho_{\lambda}(x) := (1+\delta)^{-1}(\pi_{\lambda}^{-1} + \eta \otimes I)$ is ucp. So we have ucp maps $\varphi_{\lambda}: E \to S_{\lambda}$, $\psi_{\lambda}: S_{\lambda} \to E_{\lambda}$, and $\rho_{\lambda}: E_{\lambda} \to E$. Since $\delta \to 0$ as $\epsilon \to 0$, we see that the composition $\rho_{\lambda} \circ \psi_{\lambda} \circ \varphi_{\lambda}: E \to E$ also converges to the identity map pointwise. So $\cS$ has the minimal AP.

    Now assume that $\cS$ has the minimal AP. We will show that $\cS$ is (min,el)-nuclear and hence that $\cS$ has the IEAP. To this end, it suffices to show that each finite dimensional operator subsystem $E \subseteq S$ is (min,el)-nuclear. This is because the min and el tensor products are both left-injective, and thus, for any operator system $\cT$, $x = \sum_{i=1}^n s_i \otimes t_i \in M_n(\cS \otimes_{\tau} \cT)^+$ if and only if $x \in M_n(E \otimes_{\tau} \cT)$ for some finite-dimensional subsystem $E \subseteq \cS$ such that $s_i \in E$ for each $i$. (Recall that tensor products are defined as operator system structures on the algebraic tensor product, not their completion.) 
    
    Fix $E \subseteq \cS$. Since $\cS$ has the minimal CPAP, there exist a net $\Lambda$, integers $n_{\lambda} \in \mathbb{N}$, $n_{\lambda}$-minimal operator systems $S_{\lambda}$, and ucp maps $\varphi_{\lambda}: E \to S_{\lambda}$ and $\psi_{\lambda}: S_{\lambda} \to E$ such that the composition $\psi_{\lambda} \circ \varphi_{\lambda}$ converges to the identity map on $E$ pointwise. By \cite{Kavruk}, each $S_{\lambda}$ is (min,el)-nuclear. Consider the following diagram:

    \[
    \begin{tikzcd}
        & S_{\lambda} \otimes_{\min} \cT = S_{\lambda} \otimes_{el} \cT \arrow[rd, "\psi_{\lambda} \otimes id" black] \\
        E \otimes_{min} \cT \arrow[rr] \arrow[ru, "\varphi_{\lambda} \otimes id"]  & & E \otimes_{el} \cT \subseteq \cS \otimes_{el} \cT
        \end{tikzcd}
    \]

    By functoriality of min and el tensor products, this is a diagram of ucp maps. Since the composition $\psi_{\lambda} \circ \varphi_{\lambda}$ converges to the identity map on $E$, the same holds for the composition $(\psi_{\lambda} \circ \varphi_{\lambda}) \otimes id$. Hence the identity map from $E \otimes_{min} \cT \to E \otimes_{el} \cT$ is ucp. It follows that $E \otimes_{min} \cT = E \otimes_{el} \cT$. We conclude that $\cS$ is (min,el)-nuclear and hence it has the IEAP.
\end{proof}

\textbf{Remark}: From the proof, we see that we can assume without loss of generality that the $n_{\lambda}$-minimal systems $S_{\lambda}$ all arise as subsystems of matrix algebras, i.e. $S_{\lambda} \subseteq M_{n_{\lambda}}$. Thus, each finite dimensional subsystem of a (min,el)-nuclear operator system can be ``approximated'' (as an operator system) by a subsystem of a large enough matrix algebra. 

\subsection{(min,er)-nuclearity}

In the case when $\mathcal{S}$ is finite-dimensional, we can use the minimal AP to obtain a dual characterization for $(min,er)$-nuclearity. We say that a finite-dimensional operator system $\mathcal{S}$ has the \textbf{maximal approximation property} (or \textit{maximal AP}) if there exists a net $\Lambda$, $k_{\lambda}$-maximal operator systems $\cS_{\lambda}$, and ucp maps $\varphi_{\lambda}: \cS \to \cS_{\lambda}$, $\psi_{\lambda}: \cS_{\lambda} \to \cS$ such that $\psi_{\lambda} \circ \varphi_{\lambda}: \cS \to \cS$ converges to the identity map pointwise.

\begin{figure}[h!]
\centering
    \begin{tikzcd}
    & S_{\lambda}^{k_{\lambda}-max} \arrow[rd, "\psi_{\lambda}" black] \\
    \cS \arrow[rr, "id"] \arrow[ru, "\varphi_{\lambda}"]  & & \cS
    \end{tikzcd}
\caption{The maximal approximation property.}
\end{figure}

\begin{corollary} \label{Cor: max CPAP}
    A finite-dimensional operator system $\cS$ is (min,er)-nuclear if and only if it has the maximal AP.
\end{corollary}

\begin{proof}
    %{\color{blue} This follows from Theorem \ref{thm: minimal CPAP} and Theorem 6.6 of \cite{Kavruk}, and the fact that the dual of a $k$-minimal system is $k$-maximal. In fact, we can take the $k$-maximal systems to be quotients of matrix algebras.}

    By Theorem 6.6 of \cite{Kavruk}, the operator system $\cS$ is (min,er)-nuclear if and only if its dual $\cS^d$ is (min,el)-nuclear. By Theorem \ref{thm: minimal CPAP}, this holds if and only if $\cS^d$ has the minimal AP. Furthermore, by Theorem 9.9 of \cite{Kavruk}, an operator system $\cR$ is $k$-maximal if and only if $\cR^d$ is $k$-minimal. Dualizing the diagram
    \[
        \begin{tikzcd}
        & S_{\lambda} \arrow[rd, "\psi_{\lambda}" black] \\
        \cS^d \arrow[rr, "id"] \arrow[ru, "\varphi_{\lambda}"]  & &  \cS^d
        \end{tikzcd}
    \]
    where $S_{\lambda}$ is $d_{\lambda}$-minimal and noting that $S_{\lambda}^d$ is $d_{\lambda}$-maximal, we see that $\cS$ has the minimal AP if and only if $\cS^d$ has the maximal AP.
\end{proof}

We now consider the general case of a (possibly infinite-dimensional) operator system $\cS$. To do this, we will make use of the notion of \textit{CP-stability} studied by I. Goldbring and the last author. An operator system $\cS$ is \textbf{CP-stable} if for every $\epsilon > 0$ and every finite-dimensional subsystem $E \subseteq \cS$, there exists a triple $(F,k,\delta)$ where $E \subseteq F \subseteq \cS$ is another finite-dimensional subsystem of $\cS$, $k$ is a positive integer, and $\delta > 0$ is a constant with the following property: if $\varphi: F \to \cA$ is a unital linear map from $F$ to a C*-algebra $\cA$ such that $\|\varphi^{(k)}\| < 1 + \delta$, then there exists a ucp map $\psi: E \to \cA$ such that $\|\psi - \varphi|_{E} \| < \epsilon$. In other words, approximately $k$-positive maps on $F$ restrict to approximately completely positive maps on $E$. It was shown in \cite{SinCP} that an operator system is CP-stable if and only if it has the local lifting property, or equivalently (by \cite{KPTT2}) is (min,er)-nuclear. For the convenience of the reader, we give a brief proof of the direction needed in our agruments: that LLP implies CP-stable.

\begin{proposition}
    Let $\cS$ be an operator system, and suppose that $\cS$ has the local lifting property. Then $\cS$ is CP stable.
\end{proposition}

\begin{proof}
    We only prove the infinite-dimensional case. The finite-dimensional case can be proven with a similar argument and is simpler. 
    
    For the sake of contradiction, suppose that $\cS$ is not CP-stable. Then there exists $\epsilon_0 > 0$, a finite-dimensional subsystem $E \subseteq \cS$, and a C*-algebra $\cA$ such that for every finite-dimensional operator system $F$ with $E \subseteq F \subseteq \cS$ and $n \in \mathbb{N}$, there exists a unital self-adjoint linear map $\phi_F: F \to \cA$ such that $\|\phi_F^{(n)}\| < 1 + \frac{1}{n}$ and $\| \phi_F|_{E} - \psi\| \geq \epsilon_0$ for every ucp map $\psi: E \to \cA$. In particular, we may take $n = \dim(F)$. Let $\{F_{\lambda}\}_{\lambda \in \Lambda}$ be the net of all finite-dimensional operator systems such that $E \subseteq F_{\lambda} \subseteq \cS$ ordered by inclusion and let $\varphi_{\lambda}: F_{\lambda} \to \cA$ be the corresponding maps. Let $\omega$ be a non-principle ultrafilter cofinal on $\Lambda$ and consider
    \[ \varphi: \cS_0 \to \mathcal{A}^{\omega} := (\prod_{\lambda \in \Lambda} \cA) / J_{\omega} \]
    defined by $\varphi(x) = \lim \limits_{\lambda \to \omega} \phi_{\lambda}(x)$ for $x \in \cS := \bigcup_{\lambda \in \Lambda} F_{\lambda}$, where $J_{\omega}$ denotes all nets $(a_{\lambda})$ with $\lim \limits_{\lambda \to \omega} \|a_{\lambda}\| = 0$. Then $\varphi$ is completely contractive since $\|\varphi^{(k)}\| = \lim \limits_{\lambda \to \omega} \|\phi_{\lambda}^{(k)}\| = 1$ for every $k$. It follows that $\varphi$ is completely positive. Since $\cS$ has the local lifting property, the restriction of $\varphi$ to $E$ admits a ucp lift $\psi: E \to \prod_{\lambda \in \Lambda} \cA$. Compressing $\psi$ to terms of $\prod_{\lambda \in \Lambda} \cA$, we obtain a net of maps ucp $\psi_{\lambda}: E \to \cA$ such that $\|\psi_{\lambda} - \phi_{\lambda} \| \to 0$ as $\lambda \to \omega$. This contradicts the assumption that $\|\psi - \phi_{\lambda}\| \geq \epsilon_0$ for every ucp map $\psi$. Therefore $\cS$ is CP-stable.
\end{proof}

We say that an operator system $\cS$ has the \textbf{maximal approximation property} provided that whenever $E \subseteq \cS$ is a finite-dimensional subsystem, the inclusion $i: E \to \cS$ approximately factors through finite-dimensional $k$-maximal operator systems, i.e. there exists a net of integers $k_{\lambda}$, $k_{\lambda}$-maximal and finite-dimensional operator systems $F_{\lambda}$, and ucp maps $\varphi_{\lambda}: E \to F_{\lambda}$ and $\psi_{\lambda}: F \to \cS$ such that the composition $\psi_{\lambda} \circ \varphi_{\lambda}$ converges to the inclusion $i: E \to \cS$ point-norm.

\begin{figure}[h!]
\centering
    \begin{tikzcd}
    & F_{\lambda}^{k_{\lambda}-max} \arrow[rd, "\psi_{\lambda}" black] \\
    E \arrow[rr, "i"] \arrow[ru, "\varphi_{\lambda}"]  & & \cS
    \end{tikzcd}
\caption{The maximal approximation property.}
\end{figure}

\begin{theorem}
    An operator system $\cS$ has the maximal AP if and only if it is (min,er)-nuclear.
\end{theorem}

\begin{proof} First suppose that $\cS$ has the maximal AP. Let $x \in \cS \otimes_{min} B(H)^+$. We will show that $x \in \cS \otimes_{max} B(H)^+$. Since $x$ is an element of the algebraic tensor product, there is some finite-dimensional subsystem $E \subseteq \cS$ such that $x \in E \otimes B(H)$. By injectivity of the min tensor product, $x \in E \otimes_{min} B(H)$. Consider the diagram
    \[
    \begin{tikzcd}
        & F_{\lambda} \otimes_{\min} B(H) = F_{\lambda} \otimes_{max} B(H) \arrow[rd, "\psi_{\lambda} \otimes id" black] \\
        E \otimes_{min} B(H) \arrow[rr] \arrow[ru, "\varphi_{\lambda} \otimes id"]  & & \cS \otimes_{max} B(H)
        \end{tikzcd}
    \]
Each map is completely positive, and the equality $F_{\lambda} \otimes_{\min} B(H) = F_{\lambda} \otimes_{max} B(H)$ holds by Corollary \ref{Cor: max CPAP} together with Theorem 8.5 of \cite{KPTT2}. It follows that $x \in \cS \otimes_{max} B(H)^+$. Similar arguments show that $M_n(\cS \otimes_{min} B(H))^+ = M_n(\cS \otimes_{max} B(H))^+$. So by Theorem 8.5 of \cite{KPTT2} again, $\cS$ has LLP and hence is (min,er)-nuclear.

Now suppose $\cS$ is (min,er)-nuclear. Then $\cS$ is CP stable. Let $E \subseteq \cS$ be a finite-dimensional subsystem. Let $0 < \epsilon < 1/4$ be given and let $(F=F_{\epsilon},k,\delta)$ be the corresponding triple (from CP-stability). Consider the identity map $\rho: F \to F^{k-max} \subseteq \cA$ from $F$ to $F^{k-max}$ regarded as a subsystem of some C*-cover $\cA$. The map $\rho$ is unital and linear and satisfies $\|\rho^{(k)}\| = 1 < 1+\delta$. By CP-stability, there exists a ucp map $\sigma: E \to \cA$ such that $\|\sigma - \rho|_E\| < \epsilon$. 

Since $\rho$ is isometric, $\rho(E)$ is isometrically isomorphic $E$ regarded as a subsystem of $F^{k-max}$. However, the image of $E$ under $\sigma$ may differ from $\rho(E)$. To obtain maps that factor through $F_{\epsilon}^{k-max} \subseteq \cA$, we introduce another map to perturb the range of $\sigma$. By applying Lemma \ref{lem: inverse approx cp} to the mapping $\rho(x) \mapsto \sigma(x)$ from $\rho(E)$ to $\sigma(E)$, there exists a positive linear functional $\eta: \sigma(E) \to \mathbb{C}$ with $\eta(I) \leq C \epsilon$ (where $C$ depends only on the Banach space $E$) such that $\pi + \eta \otimes I$ is completely positive, where $\pi: \sigma(E) \to \rho(E)$ is defined by $\pi: \sigma(x) \mapsto \rho(x)$ for $x \in E$. Renormalizing, we obtain a map $\varphi_{\epsilon}: E \to F_{\epsilon}^{k-max}$ given by
$\varphi_{\epsilon} = (1 + \eta(I))^{-1}(\pi + \eta \otimes I) \circ \sigma = (1 + \eta(I))^{-1}(\rho + (\eta \circ \sigma) \otimes I)$. Since $\pi + \eta \otimes I$ is completely positive and $\sigma$ is ucp, $\varphi_{\epsilon}$ is ucp. Let $\psi_{\epsilon}: F_{\epsilon}^{k-max} \to \cS$ be the inclusion map. Since $k$-positive maps on $k$-maximal systems are completely positive \cite{Xhabli}, $\psi_{\epsilon}$ is ucp. Finally, the composition $\psi_{\epsilon} \circ \varphi_{\epsilon}$ converges to the inclusion map as $\epsilon \to 0$. We conclude that $\cS$ has the maximal AP. \end{proof}

\subsection{(min,c)-nuclearity}

An operator system has the \textbf{universal cover approximation property} (UCAP) if the inclusion $j: \cS \to C^*_u(\cS)$ approximately factors through matrix algebras --- i.e. there exists a net $\Lambda$ and ucp maps $\varphi_{\lambda}: \cS \to M_{n_{\lambda}}$ and $\psi_{\lambda}: M_{n_{\lambda}} \to C^*_u(\cS)$ such that for every $x \in \cS$, $\psi_{\lambda} \circ \varphi_{\lambda} (x) \to j(x)$.

\begin{figure}[h!]
\centering
    \begin{tikzcd}
    & M_{n_{\lambda}} \arrow[rd, "\psi_{\lambda}" black] \\
    \mathcal{S} \arrow[rr, "j"] \arrow[ru, "\varphi_{\lambda}"]  & & C^*_u(\cS)
    \end{tikzcd}
\caption{The universal cover approximation property.}
\end{figure}
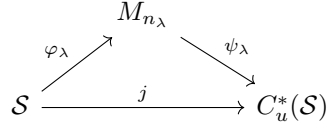

\begin{theorem} \label{t:min-c-nuc}
    Let $\cS$ be an operator system. The the following statements are equivalent.
    \begin{enumerate}
        \item $\cS$ is (min,c)-nuclear.
        \item $\cS$ is C*-nuclear.
        \item $\cS$ has the UCAP.
    \end{enumerate}
\end{theorem}

\begin{proof}
    The equivalence of (1) and (2) follows from \cite{KPTT1}. To see that (1) and (3) are equivalent, Let $\cT = C^*_u(\cS)$ and let $\varphi: \cS \to \cT$ denote the canonical inclusion. By Theorem \ref{thm: Han-Paulsen}, $id \otimes \varphi: R \otimes_{min} \cS \to R \otimes_{max} \cT$ is ucp for every operator system $R$ if and only if (3) holds. But the image $id \otimes \varphi( R \otimes_{min} \cS) \subseteq R \otimes_{max} \cT$ is exactly $R \otimes_c \cS$. It follows that (1) and (3) are equivalent.
\end{proof}

We will also consider the following approximate factorization property, which we will see is equivalent. We say that an operator system $\mathcal{S}$ has the \textit{C*-cover approximation property} if and only if for every ucp map $\pi: \cS \to B(H)$ there exists a net $\Lambda$ and ucp maps $\varphi_{\lambda}: \cS \to M_{n_{\lambda}}$ and $\psi_{\lambda}: M_{n_{\lambda}} \to C^*(\pi(\cS))$ such that for every $x \in \cS$, $\psi_{\lambda} \circ \varphi_{\lambda} (x) \to \pi(x)$.

\begin{figure}[h!]
\centering
    \begin{tikzcd}
    & M_{n_{\lambda}} \arrow[rd, "\psi_{\lambda}" black] \\
    \mathcal{S} \arrow[rr, "\pi"] \arrow[ru, "\varphi_{\lambda}"]  & & C^*(\pi(\cS))
    \end{tikzcd}
\caption{The C*-cover approximation property.}
\end{figure}
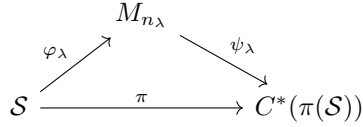

\begin{lemma}
    Let $\cS$ be an operator system. Then $\cS$ has the UCAP if and only if it has the C*-cover AP.
\end{lemma}

\begin{proof}
    Suppose that $\cS$ has the UCAP. Then there exists a net $\Lambda$ and ucp maps $\varphi_{\lambda}: \cS \to M_{n_{\lambda}}$ and $\psi_{\lambda}: M_{n_{\lambda}} \to C^*_u(\cS)$ such that for every $x \in \cS$, $\psi_{\lambda} \circ \varphi_{\lambda} (x) \to j(x)$. Let $\pi: \cS \to B(H)$ be ucp. By the universal property of $C^*_u(\cS)$, there exists a $*$-homomorphism $\rho: C^*_u(\cS) \to C^*(\pi(\cS))$ such that $\rho \circ j = \pi$. Hence we have a net of maps $\varphi_{\lambda}: \cS \to M_{n_{\lambda}}$ and $\psi_{\lambda}' = \rho \circ \psi_{\lambda}: M_{n_{\lambda}} \to C^*(\pi(\cS))$ such that for every $x \in \cS$, $\psi_{\lambda}' \circ \varphi_{\lambda}(x) \to \pi(x)$. The other direction is trivial, since the statement must be true for $\pi = j: \cS \to C^*_u(\cS) \subseteq B(H)$.
\end{proof}

Let $G = (V,E)$ be a finite undirected graph with self-loops at each vertex. Writing $V = \{1,2,\dots,n\}$ for the set of vertices of $G$ and $E \subseteq V \times V$ for the set of vertices, we associate to $G$ an operator syste $\cS_G \subseteq M_n$ defined by \[ \cS_G := \text{span} \{ E_{ij} : (i,j) \in E\} \] where $E_{ij}$ denotes the $ij$ matrix unit in $M_n$. We refer to any operator system of this form as a \textit{graph operator system}.

Tensor properties of graph operator systems were considered in \cite{KPTT1}. It was shown in that whenever $G$ is a chordal graph (i.e. every cycle of length 4 or greater has a chord), then $\cS_G$ is (min,c)-nuclear. The proof relied on a matrix completion result from \cite{PaulsenRodman} which applies only to chordal graphs. Using the UCAP, we show that such matrix completions characterize all (min,c)-nuclear graph operator systems.

\begin{theorem} \label{thm: graph system min-c}
    Let $G$ be a finite undirected graph with self-loops at each vertex and $\cS_G$ its graph operator system. Then $\cS_G$ is (min,c)-nuclear if and only if for every C*-algebra $\cA$ and ucp map $\varphi: \cS_G \to A$ and every $\epsilon > 0$, there exists a matrix $(a_{ij}) \in M_n(A)$ such that $(a_{ij}) + \epsilon I_n \in M_n(A)^+$ and $\varphi(E_{ij}) = a_{ij}$ for all $(i,j) \in E$. In other words, $\cS_G$ is (min,c)-nuclear if and only if the partially-defined matrix $(\varphi(E_{ij}))_{(i,j) \in E} + \epsilon I_n$ has a positive completion in $\cA$ for every $\epsilon > 0$.
\end{theorem}

\begin{proof}
    Suppose $\cS_G$ is (min,c)-nuclear. Then for every C*-algebra $\cA$ and ucp map $\pi: \cS_G \to \cA$, there exists a net $\Lambda$ and ucp maps $\varphi_{\lambda}: \cS \to M_{n_{\lambda}}$ and $\psi_{\lambda}: M_{n_{\lambda}} \to \cA$ such that for every $x \in \cS$, $\psi_{\lambda} \circ \varphi_{\lambda} (x) \to j(x)$. Since $M_{n_{\lambda}}$ is injective, we may extend $\varphi_{\lambda}$ to a ucp map $\varphi_{\lambda}': M_n \to M_{n_{\lambda}}$.

        \[
        \begin{tikzcd}
        M_n \arrow[r, "\varphi_{\lambda}'"] & M_{n_{\lambda}} \arrow[rd, "\psi_{\lambda}" black] \\
        \cS_G \arrow[u, hook] \arrow[rr, "\pi"] \arrow[ru, "\varphi_{\lambda}"]  & & A
        \end{tikzcd}
        \]

    \noindent Since $\pi_{\lambda} = \psi_{\lambda} \circ \varphi_{\lambda}': M_n \to \cA$ is ucp, the matrix $\pi_{\lambda}(E_{ij})$ is a positive matrix in $M_n(\cA)$, and $\pi_{\lambda}(E_{ij}) \to \pi(E_{ij})$ for each $(i,j) \in E$. For each $\lambda$, define $a_{\lambda} = (a_{ij}^{\lambda})$ where $a_{ij}^{\lambda} = \pi_{\lambda}(E_{ij})$ for all $(i,j)$. Then for each $\epsilon > 0$, there exists $\lambda$ such that $\|(\pi_{\lambda}(E_{ij})) - a_{\lambda} \| < \epsilon$, implying $(\pi_{\lambda}(E_{ij})) - a_{\lambda} \leq \epsilon I_n$ and hence $0 \leq (\pi_{\lambda}(E_{ij})) \leq \epsilon I_n + a_{\lambda}$. The matrix $a_{\lambda}$ satisfies the required conditions, since $a_{ij}^{\lambda} = \pi(E_{ij})$ when $(i,j) \in G$.

    Conversely, suppose for every C*-algebra $\cA$ and ucp map $\varphi: \cS_G \to A$ and every $\epsilon > 0$, there exists a matrix $(a_{ij}) \in M_n(A)$ such that $(a_{ij}) + \epsilon I_n \in M_n(A)^+$ and $\varphi(E_{ij}) = a_{ij}$ for all $(i,j) \in E$. Let $\pi: \cS_G \to A$ be ucp and let $\epsilon > 0$. Then the map $\pi_{\epsilon}: M_n \to A$ given by $\pi_{\epsilon}(E_{ij}) = a_{ij} + \epsilon \delta_{ij} I$ is ucp, by Choi's Theorem. Since $\pi$ is ucp, $\pi_{\epsilon}(I) = (1+n\epsilon) I$. Hence $\psi_{\epsilon} := (1+n \epsilon)^{-1} \pi_{\epsilon}: M_n \to A$ is ucp. Letting $\varphi: \cS_G \to M_n$ denote the inclusion map, we obtain a sequence of ucp maps such that $\psi_{1/k} \circ \varphi \to \pi$ pointwise as $\epsilon \to 0$. So $\cS_G$ has the UCAP and hence is (min,c)-nuclear.
\end{proof}

\subsection{Weak-$*$ approximate factorization properties}

%\color{blue} Maybe combine this section with the next one. \color{black} 
Let $\cS$ be an operator system. We say $\cS$ has the \textbf{injective double-commutant approximation property} (injective DCAP) if for any unital complete order embedding $\cS \subseteq B(H)$, the inclusion $j: \cS \to \cS'' \subseteq B(H)$ approximately factors through injective operator systems --- i.e. there exists a net of injective operator systems $\{I_{\lambda}\}$ and ucp maps $\varphi_{\lambda}: \cS \to I_{\lambda}$ and $\psi_{\lambda}: I_{\lambda} \to \cS''$ such that for every $x \in \cS$, $\psi_{\lambda} \circ \varphi_{\lambda} (x) \to j(x)$ in the point weak-$*$ topology.

\begin{figure}[h!]
\centering
\begin{tikzcd}
    & I_{\lambda} \arrow[rd, "\psi_{\lambda}" black] \\
    \mathcal{S} \arrow[rr, "j"] \arrow[ru, "\varphi_{\lambda}"]  & & \cS''
    \end{tikzcd}
\caption{The injective double-commutant approximation property}
\end{figure}

\begin{theorem} \label{thm: DCEP iff double commutant CPAP}
    Let $\cS$ be an operator system. The the following statements are equivalent.
    \begin{enumerate}
        \item $\cS$ is (el,c)-nuclear.
        \item $\cS$ has DCEP.
        \item $\cS$ has the injective DCAP.
    \end{enumerate}
\end{theorem}

\begin{proof}
The equivalence of (1) and (2) is proved in \cite{KPTT2}. That (2) implies (3) is clear, since DCEP implies the existence of a ucp map $\pi: I(\cS) \to \cS''$ --- thus, (3) holds for a fixed injective system $I(\cS)$, taking $\varphi: \cS \to I(\cS)$ to be the inclusion and $\pi: I(\cS) \to \cS''$ to be the given ucp map.

To show (3) implies (2), extend each $\varphi_{\lambda}$ to ucp maps $\varphi_{\lambda}': I(\cS) \to I_{\lambda}$ by injectivity of $I_{\lambda}$. Then we have a net of maps $\psi_{\lambda} \circ \varphi_{\lambda}'$ from $I(\cS)$ to $\cS''$. Since $\cS''$ has a predual, there exists a point weak-$*$ limit $\varphi: I(\cS) \to \cS''$ extending the inclusion $j$, so $\cS$ has DCEP.

\[
\begin{tikzcd}
I(\cS) \arrow[r, "\varphi_{\lambda}'"] & I_{\lambda} \arrow[rd, "\psi_{\lambda}"] \\
\mathcal{S} \arrow[u, "i", hook] \arrow[rr, "j"] \arrow[ru, "\varphi_{\lambda}"]  & & \cS''
\end{tikzcd}
\]

\end{proof}

Let $\cS$ be an operator system and let $\cS^{**}$ denote the operator system bidual of $\cS$. We say $\cS$ has the \textbf{injective bidual approxmation property} (injective BAP) if the inclusion $j: \cS \to \cS^{**}$ approximately factors through injective operator systems --- i.e. there exists a net $I_{\lambda}$ of injective operator systems and ucp maps $\varphi_{\lambda}: \cS \to I_{\lambda}$ and $\psi_{\lambda}: I_{\lambda} \to \cS^{**}$ such that for every $x \in \cS$, $\psi_{\lambda} \circ \varphi_{\lambda} (x) \to j(x)$ point weak-$*$.

\begin{figure}[h!]
\centering
    \begin{tikzcd}
    & I_{\lambda} \arrow[rd, "\psi_{\lambda}" black] \\
    \mathcal{S} \arrow[rr, "j"] \arrow[ru, "\varphi_{\lambda}"]  & & \cS^{**}
    \end{tikzcd}
\caption{The injective bidual approximation property.}
\end{figure}
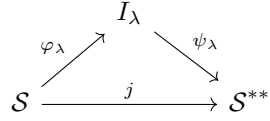

\begin{theorem} \label{t:el-max-nuc}
    Let $\cS$ be an operator system. The the following statements are equivalent.
    \begin{enumerate}
        \item $\cS$ is (el,max)-nuclear.
        \item $\cS$ has WEP.
        \item $\cS$ has the injective BAP.
    \end{enumerate}
\end{theorem}

\begin{proof}
    The implication (2) implies (1) was shown in \cite{KPTT2}, while the implication (1) implies (2) was shown by Han in \cite{Htp}. The proof that (2) is equivalent to (3) is similar to the proof of Theorem \ref{thm: DCEP iff double commutant CPAP}.
\end{proof}

 Again, let $\cS$ be an operator system and let $\cS^{**}$ denote the operator system bidual of $\cS$. We say $\cS$ has the \textbf{C*-bidual approxmation prorty} (C*-BAP) if the inclusion $j: \cS \to \cS^{**}$ approximately factors through C*-algebras --- i.e. there exists a net $A_{\lambda}$ of C*-algebras and ucp maps $\varphi_{\lambda}: \cS \to A_{\lambda}$ and $\psi_{\lambda}: A_{\lambda} \to \cS^{**}$ such that for every $x \in \cS$, $\psi_{\lambda} \circ \varphi_{\lambda} (x) \to j(x)$ point weak-$*$.

 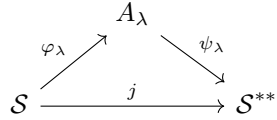
\begin{figure}[h!]
\centering
    \begin{tikzcd}
    & A_{\lambda} \arrow[rd, "\psi_{\lambda}" black] \\
    \mathcal{S} \arrow[rr, "j"] \arrow[ru, "\varphi_{\lambda}"]  & & \cS^{**}
    \end{tikzcd}
\caption{The C*-bidual approximation property.}
\end{figure}

 Since every injective operator system is completely order isomorphic to a C*-algebra, it is clear that WEP implies the C*-CPAP to bidual. This is also seen by the following nuclearity characterization.

 \begin{theorem} \label{t:c-max-nuc}
    Let $\cS$ be an operator system. The the following statements are equivalent.
    \begin{enumerate}
        \item $\cS$ is (c,max)-nuclear.
        \item $\cS$ has the C*-BAP.
    \end{enumerate}
\end{theorem}

\begin{proof}
    By Theorem 2.3 of \cite{KavrukRWI} (and its proof), we see that $\cS$ is (c,max)-nuclear if and only if it is relatively weakly injective in $C^*_u(\cS)$, i.e. the inclusion $i^{**}: \cS \to \cS^{**}$ extends to a ucp map $\psi: C^*_u(\cS) \to \cS^{**}$. Thus, if $\cS$ is (c,max)-nuclear, we may take $\varphi: \cS \to C^*_u(\cS)$ to be the inclusion and $\psi: C^*_u(\cS) \to \cS^{**}$ to be an extension, so that $\cS$ trivially satisfies the C*-BAP. On the other hand, suppose that $\cS$ satisfies the C*-BAP with a net of maps $\varphi_{\lambda}: \cS \to A_{\lambda}$ and $\psi_{\lambda}: A_{\lambda} \to \cS^{**}$. By the universal property of $C^*_u(\cS)$, each ucp map $\varphi_{\lambda}$ extends to a $*$-homomorphism $\pi_{\lambda}: C^*_u(\cS) \to A_{\lambda}$. Since $\cS^{**}$ is a dual operator system, the net of maps $\psi_{\lambda} \circ \pi_{\lambda}: C^*_u(\cS) \to \cS^{**}$ has a weak-$*$ limit point $\rho: C^*_u(\cS) \to \cS^{**}$ in the compact space of all ucp maps from $C^*_u(\cS)$ to $\cS^{**}$. Since $\psi_{\lambda} \circ \pi_{\lambda}(s) = \psi_{\lambda} \circ \varphi_{\lambda}(s) \to s$ weak-$*$, we conclude that $\rho$ extends the inclusion of $\cS$ into $\cS^{**}$. So $\cS$ is relatively weakly injective in $C^*_u(\cS)$ and therefore $\cS$ is (c,max)-nuclear.
\end{proof}

Let $\cA$ be a dual operator system (i.e. an operator system with a Banach space predual), and suppose that $i: \cS \to \cA$ is a complete order embedding. We say that $i: \cS \to \cA$ is \textit{semi-discrete} if there exists a net of integers $d_{\lambda}$ and ucp maps $\varphi_{\lambda}: \cS \to M_{d_{\lambda}}$ and $\psi_{\lambda}: M_{d_{\lambda}} \to \cA$ such that the composition $\psi_{\lambda} \circ \varphi_{\lambda} \to i$ point weak-$*$. The next statement rephrases the injective DCAP and injective BAP in terms of semi-discrete inclusions in the case when the intermediate injective operator systems are all matrix algebras.

\[
\begin{tikzcd}
& M_{d_{\lambda}} \arrow[rd, "\psi_{\lambda}" black] \\
\mathcal{S} \arrow[rr, "i"] \arrow[ru, "\varphi_{\lambda}"]  & & \cA
\end{tikzcd}
\]

\begin{corollary} \label{Cor: semidiscrete implies nuclear}
    Suppose that $\cS$ and $\mathcal{A}$ are operator systems and $i: \cS \to \cA$ is a complete order embedding.
    \begin{enumerate}
        \item If $i$ is semidiscrete whenever $\cA = \pi(\cS)''$ and $\pi: \cS \to B(H)$ is a complete order embedding, then $\cS$ is (el,c)-nuclear.
        \item If $i$ is semidscrete whenever $\cA = \cS^{**}$, then $\cS$ is (el,max)-nuclear.
    \end{enumerate}
\end{corollary}

The notion of semidiscrete inclusions of operator systems was studied by Le Merdy in \cite{LeMer}. Theorem 4.5 of \cite{LeMer} in particular implies that the inclusion $i: \cS \to \pi(\cS)''$ is semidiscrete if and only if the natural map $\cS \otimes_{min} T \to \pi(\cS)'' \otimes_{nor} T$ is a complete order embedding. The tensor product $\otimes_{nor}$ is defined as follows: Given a von Neumann algebra $\mathcal{M}$ and an operator system $\cT$ and $x \in \mathcal{M} \otimes \cT$ (the algebraic tensor product), we define \[ \|x\| = \sup \{ \| (\rho \cdot \psi)(x) \| : \rho: \mathcal{M} \to B(H) \text{ is a normal $*$-representation and } \psi: \cT \to \rho(M)' \text{ is ucp} \}. \]
Matrix norms are defined analogously, and the resulting operator space can be identified completely isometrically with an operator system in an obvious way. The tensor product $\otimes_{nor}$ was defined by Effros and Lance in \cite{EL} in the context of semidiscreteness for von Neumann algebras. 

The converse of Corollary~\ref{Cor: semidiscrete implies nuclear} (1) fails in general: requiring semidiscrete factorizations through matrix algebras for every representation is equivalent to the stronger property of (min,c)-nuclearity.
In fact, the $I_{\lambda}$'s in the definition of the injective DCAP can be replaced with matrix algebras if and only if the system is (min,c)-nuclear.

\begin{theorem} \label{thm: weak* (min,c)}
    The inclusion $\pi: \cS \to \pi(\cS)''$ is semidiscrete for every complete order embedding $\pi: \cS \to B(H)$ if and only if $\cS$ is (min,c)-nuclear.
\end{theorem}

\begin{proof}
    If $\cS$ is (min,c)-nuclear, then it has the UCAP. It follows from the universal property of $C^*_u(\cS)$ that there exist nets of maps $\varphi_{\lambda}: \cS \to M_{n_{\lambda}}$ and $\psi_{\lambda}: M_{n_{\lambda}} \to C^*(\pi(\cS)) \subseteq \pi(\cS)''$ converging point-norm to the inclusion $\pi: \cS \to \pi(\cS)''$ for every complete order embedding $\pi: \cS \to B(H)$.

    Conversely, suppose $\pi: \cS \to \pi(\cS)''$ is semidiscrete for every complete order embedding $\pi: \cS \to B(H)$. Consider the universal representation $\pi_u: \cS \to C^*_u(\cS)$. Then $\pi_u(\cS)'' = C^*_u(\cS)^{**}$. By Le Merdy's Theorem 4.5 \cite{LeMer}, for any operator system $\cT$, the natural map $\cS \otimes_{min} \cT \to C^*_u(\cS)^{**} \otimes_{nor} T$ is a unital complete order embedding. However, by the definition of the tensor product $C^*_u(\cS)^{**} \otimes_{nor} \cS$, we also have that the natural map $C^*_u(\cS) \otimes_c \cT \to C^*_u(\cS)^{**} \otimes_{nor} \cT$ is a unital complete order embedding, since ucp maps on $C^*_u(\cS)$ extend to normal ucp maps on $C^*_u(\cS)^{**}$. Since the image of the map $\cS \otimes_{min} \cT \to C^*_u(\cS)^{**} \otimes_{nor} \cT$ is exactly $\cS \otimes_c \cT \subseteq C^*_u(\cS) \otimes_c \cT$, it follows that $\cS$ is (min,c)-nuclear.
\end{proof}

\subsection{Self-nuclearity} 

We conclude by considering ``self-nuclearity'' for finite-dimensional operator systems $\cS$. Given $\mathcal S$ then define the relative order structure $\tau$ as that induced via the embedding \[
\mathcal S \otimes_{\tau} \mathcal S \hookrightarrow \mathcal S \otimes_{\text{max}} \mathcal A,
\] where $\tau$ is the relative operator system structure, and $\mathcal A$ is an operator system containing $\mathcal S$, completely order isomorphically. 

\begin{theorem}
A finite-dimensional operator system $\mathcal S$ satisfies \[\mathcal S \otimes_{\text{min}} \mathcal S = S \otimes_{\tau} \mathcal S\] if and only if for every $n \in \mathbb N$ and completely positive map $u: \mathcal S^d \to  M_n(\mathcal S)$, there exists completely positive maps $\varphi_\lambda: \mathcal S^d \to  M_{n_\lambda}, \psi_\lambda:  M_{n_\lambda} \to M_n(\mathcal A)$ such that \[
\psi_\lambda \circ \varphi_\lambda \to u: \cS^d \to M_n(\cS) \subseteq M_n(\cA),
\] where convergence is in the point-norm sense. 
\end{theorem}

\begin{proof}
    First assume $\mathcal S \otimes_{min} \mathcal S = \mathcal S \otimes_\tau \mathcal S$, fix $n \in \mathbb N$, and let $u: \mathcal S^d \to M_n(\mathcal S)$ be a completely positive map. Denote the induced tensor by $t \in \mathcal S \otimes   M_n(\mathcal S)$. Complete positivity of $u$ implies $t \in [\mathcal S \otimes_{min} M_n(\mathcal S)]^+$, and by our assumptions we have \[
t \in [\mathcal S \otimes_\tau M_n(\mathcal S)]^+ \subset [ \mathcal S \otimes_{max} M_n(\mathcal A)]^+.
    \] Let $\e >0$ and express $t + \e I_n$ as $\hat{a}(P \otimes Q)\hat{a}^*$, where $P \in M_{d_1}(\mathcal S)^+, Q \in M_{d_2}(M_n(\mathcal A))^+$, and $\hat{a} = I_n \otimes a, a \in M_{1,d_1d_2}$. We then define the maps \[
\varphi_\epsilon: S^d \to M_{d_1},
\psi_\e: M_{d_1} \to M_n(\mathcal A), 
    \] where $\varphi_\epsilon$ is the induced map from $P$, and $\psi_\epsilon$ is given as $\rho \mapsto \hat{a}(\rho \otimes Q)a^*$. Without loss of generality we may assume $\varphi_\epsilon$ is a complete contraction, by scaling appropriately and combining scaling factor with the scalar matrices $a$. Then $\psi_\epsilon \circ \varphi_\epsilon: \mathcal S^d \to M_n(\mathcal A)$ is the desired approximation of $u$. 
%\newline 

Conversely, assume one has the approximate factorization property, and let $t \in M_n(\mathcal S \otimes_{min} \mathcal S)^+$. Implementing suitable min-tensor properties, we realize the tensor $t$ in the cone $[\mathcal S \otimes_{min} M_n(\mathcal S)]^+$. 
We claim that $t \in [\mathcal S \otimes_{\tau} M_n(\mathcal S)]^+$. Let $u: \mathcal S^d \to M_n(\mathcal S)$ denote the induced completely positive operator. Let $(\varphi_i), (\psi_i), \varphi_i: \mathcal S^d \to M_{n_i}, \psi_i: M_{n_i} \to M_n(\mathcal A)$ denote the nets of completely positive maps such that $\psi_i\circ \varphi_i \to u$ point-norm. As in the first part of the proof, we may take $\varphi_i: \mathcal S^d \to M_{n_i}$ to be a complete contraction. Let $t_i \in M_{n_i}^d \otimes M_n(\mathcal A)$ denote the tensor corresponding to $\psi_i$. Using the fact that $M_{n_i}^d$ is completely order isomorphic to $M_{n_i}$, and nuclearity of matrix algebras, we observe $t_i \in [M_{n_i}^d \otimes_{max} M_n(\mathcal A)]^+$. Thus, given $\epsilon > 0$ then we may express \[
t_i + \epsilon I = a(P \otimes Q)a^*, \quad P \in M_{d_1}(M_{n_i}^d)^+, \quad Q \in M_{d_2}(M_n(\mathcal A))^+,
\] and where $a$ is a rectangular scalar matrix. Let $u_\psi: M_{n_i} \to M_{d_1}$ denote the completely positive map corresponding to $P$, and consider the composition \[
\xi_i:= u_i \circ \varphi_i: \mathcal S^d \to M_{n_i} \to M_{d_1}. 
\] Finally, define the completely positive map \[
\Theta: \mathcal S^d \to M_n(\mathcal A), \quad s^d \mapsto a(\xi_i(s^d) \otimes Q)a^*. 
\]  It then follows via construction that \[
\psi_i \circ \varphi_i = \Theta - \epsilon I = a(\xi_i \otimes Q)a^* - \epsilon I.
\] Taking a limit, we conclude that $t + \epsilon I \in [\cS \otimes_{max} M_n(\cA)]^+$. This shows the desired approximation yielding $t \in [\mathcal S \otimes_\tau M_n(\mathcal S)]^+$, and we conclude \[
\text{id}: \mathcal S \otimes_{min} \mathcal S \to \mathcal S \otimes_\tau \mathcal S,
\] is completely positive. \qedhere
\end{proof}

\begin{remark}
    By the commutativity of the maximal tensor product and symmetry, the operator system structure induced by the inclusion $\cS \otimes_{\tau} \cS \subseteq \cS \otimes_{max} \cA$ agrees with the one induced by the inclusion $\cS \otimes_{\tau} \cS \subseteq \cA \otimes_{max} \cS$.
\end{remark}

\begin{corollary} \label{cor: min-c min-el self nuc}
    Given a finite-dimensional operator system $\mathcal S$, then \[
\mathcal S \otimes_{min} \mathcal S = \mathcal S \otimes_c \mathcal S,
    \] if and only if for every $n \in \mathbb N$ and every completely positive map $u: \mathcal S^d \to M_n(\mathcal S)$ there exists nets of completely positive operators $(\varphi_i)_i, (\psi_i)_i, \varphi_i: \mathcal S^d \to M_{n_i}, \psi_i: M_{n_i} \to M_n(C_u^*(\mathcal S))$ such that $\psi_i \circ \varphi_i \to u$ in the point-norm sense. Similarly, \[ \cS \otimes_{min} \cS = \cS \otimes_{el} \cS \] if and only if for every $n \in \mathbb N$ and every completely positive map $u: \mathcal S^d \to M_n(\mathcal S)$ there exists nets of completely positive operators $(\varphi_i)_i, (\psi_i)_i, \varphi_i: \mathcal S^d \to M_{n_i}, \psi_i: M_{n_i} \to M_n(I(\cS))$ such that $\psi_i \circ \varphi_i \to u$ in the point-norm sense.
\end{corollary}

\begin{remark}
    \emph{The topic of self-nuclearity was studied by Kavruk in the context of Kirchberg's QWEP conjecture. Kavruk showed (in Theorem 5.10 of \cite{Kavruk}) that Kirchberg's conjecture is true if and only if $S_2 \otimes_{min} S_2 = S_2 \otimes_c S_2$, where $S_2$ denotes the five-dimensional operator system generated by the unitary generators of the free group $G=\mathbb{F}_2$ in $C^*(G)$. By \cite{MIP}, Kirchberg's conjecture is false and therefore $S_2 \otimes_{min} S_2 \neq S_2 \otimes_c S_2$. By Corollary \ref{cor: min-c min-el self nuc}, $S_2$ lacks the approximation property described in the Corollary.}
\end{remark}

\section*{Acknowledgments}

This work was supported by the American Institute of Mathematics through the AIM SQuaRE workshop ``Approximation properties for operator systems and matrix convex sets.'' The first and third authors were partially supported by the Binational Science Foundation startup grant BSF-202416. The fourth author was partially supported by NSF grant DMS-2055155. Part of this work was completed during the third author's visit to the Mittag-Leffler Institute workshop ``Operator algebras and quantum information theory.'' The third author thanks Kristin Courtney for a helpful discussion contributing to Theorem \ref{t:c-max-nuc}. Finally, we thank Adam Dor-On and Matthew Kennedy for many helpful discussions on this work, including a key observation needed to realize Theorem \ref{thm: weak* (min,c)}.
\color{black}

\bibliographystyle{alpha}

\end{document}